\documentclass[10pt,oneside]{amsart}
\usepackage[T1]{fontenc}
\usepackage{wrapfig}

\usepackage{geometry} 
\usepackage{xcolor} 
\usepackage{tikz} 
\usetikzlibrary{cd} 

\usepackage{hyperref} 
\usepackage{amsmath,amsthm,amssymb,mathtools}
\usepackage[all,cmtip]{xy}
\usepackage{ wasysym }
\usepackage{todonotes}

\newcommand{\thh}{\operatorname{THH}}
\newcommand{\Set}{\operatorname{Set}}
\newcommand{\TR}{\operatorname{TR}}

\usepackage[capitalise]{cleveref}  
\newcommand{\clevertheorem}[3]{%
  \newtheorem{#1}[thm]{#2}
  \crefname{#1}{#2}{#3}
}

\numberwithin{equation}{section} 
\numberwithin{figure}{section} 

\theoremstyle{plain} 
\newtheorem{thm}{Theorem}[section]
\crefname{thm}{Theorem}{Theorems}
\newtheorem*{thm*}{Theorem}
\clevertheorem{prop}{Proposition}{Propositions}
\newtheorem*{prop*}{Proposition}
\clevertheorem{lem}{Lemma}{Lemmas}
\clevertheorem{cor}{Corollary}{Corollaries}
\clevertheorem{conj}{Conjecture}{Conjectures}

\theoremstyle{definition} 
\clevertheorem{defn}{Definition}{Definitions}
\clevertheorem{notn}{Notation}{Notations}

\theoremstyle{remark} 
\clevertheorem{rmk}{Remark}{Remarks}
\clevertheorem{eg}{Example}{Examples}
\clevertheorem{example}{Example}{Examples}
\clevertheorem{ex}{Exercise}{Exercises}

\makeatletter\let\c@equation\c@thm\makeatother
\makeatletter\let\c@figure\c@thm\makeatother

\crefname{figure}{Figure}{Figures}
\crefname{equation}{}{Displays} 
\crefname{eq}{}{Displays}
\crefname{eqn}{}{Displays}

\newcommand{\mbf}{\mathbf}
\newcommand{\id}{\operatorname{id}}

\newcommand{\End}{\operatorname{End}}
\newcommand{\tr}{\operatorname{Tr}}
\newcommand{\Tr}{\operatorname{Tr}}
\newcommand{\ch}{\operatorname{ch}}
\newcommand{\THH}{\operatorname{THH}}

\DeclareMathOperator{\holim}{holim} 

\hypersetup{
 pdfkeywords={latex starter,template},
 pdfauthor={Niles Johnson},
}

\title{Facets of the Witt Vectors}
\author{Jonathan A. Campbell}
\date{}
\begin{document}
\maketitle
\begin{abstract}
We give a $K$-theoretic account of the basic properties of Witt vectors. Along the way we re-prove basic properties of the little-known Witt vector norm,  give a characterization of Witt vectors in terms of algebraic $K$-theory, and a presentation of the Witt vectors that is useful for computation. 
\end{abstract}
\section{Introduction} 

The ring of Witt vectors is an object that rears its head frequently in algebra, number theory, algebraic geometry, and homotopy theory.  For all this, it seems to cause consternation because of its opaque definition and the number of equally opaque identities that accompany it. Our aim is to demystify the definition, explain its ubiquity and provide devices for working with the identities.

Let $A$ be a ring. We take as a starting point Almkvist's beautiful theorem that the group completion of the category of endomorphisms of finitely generated projective $A$-modules, $\widetilde{K}_0 (\End(A))$, is equivalent to the \textit{rational} big Witt vectors, i.e. the Witt vectors that appear as fractions of polynomials in $(1+A[[t]])^\times$ \cite{almkvist1}. Almkvist's theorem tells us that the Witt vectors are nothing mysterious: they are objects that capture the compatibilites between iterated traces of endomorphisms. In the instances where this does not literally work, the Witt vectors can be viewed as formal devices that \textit{force} this to work. Indeed, the (big) Witt vectors essentially want to be $\widetilde{K}_0 (\End(A))$, but the input category for $K$-theory has to be \textit{finite} in some way. If we allowed ourselves infinite dimensional endomorphisms, we would be able to capture all characteristc polynomials and thus all, not just rational, Witt vectors. However, this is proscribed by the construction of K-theory: an Eilenberg Swindle will not allow for infinite dimensional data. We repair this by a simple completion procedure, inspired by \cite{DS1}.

\begin{thm}\label{thm:theorem_1}
There is a filtration $V^\bullet$ on $\widetilde{K}_0 (\End(A))$ such that completing $\widetilde{K}_0 (\End(A))$ with respect to $V^\bullet$ is $W(A)$. 
\end{thm}

This statement has many applications. First, the Witt vectors do not have many pleasant universal properties, but this characterization points the way: for example, it should be an initial $V$-complete object in a suitable category. 

Further, although \cref{thm:theorem_1} is a kind of universal statement, it still allows us to extract a very concrete presentation of the Witt vectors, suitable for computation.

  \begin{thm} (see \cref{thm:presentation})
    For a commutative ring $A$, the ring of big Witt vectors $W(A)$, is isomorphic to formal sums of symbols $(a)_i$, $a \in A, i \in \mathbb{N}$, with multiplication given by the formula 
\[
   (a)_i  (b)_j = \operatorname{gcd} (i, j)(a^{j/\gcd(i,j)}  b^{i/\gcd(i,j)})_{\operatorname{lcm}(i,j)}
\]
  \end{thm}

  \cref{thm:theorem_1} also us us how to construct operations on the Witt vectors. Any functorial construction on $\End(A)$ that is continuous in the $V$-topology descends to an operation on Witt vectors. This allows for a conceptual, categorical proof of the following theorem. 

\begin{thm}\cite{angeltveit}
There is a a multiplicative norm on the (big) Witt vectors $W(A)$ induced by a norm functor $N^{i}$ on $\End(A)$ (see \cref{sec:main_section} for definitions). 
\end{thm}

This multiplicative norm is exactly the multiplicative analogue of the standard Verschiebung map.

In order to introduce this circle of ideas, we start with the characteristic polynomial. Let $V_1, V_2$ be vector spaces and $T_1, T_2$ be endomorphisms, and define the characteristic polynomial to be $\ch(T_i) = \det(\id - t T_i) \in 1 + A[t]$. The category of vector spaces is endowed with the binary operation ``$\oplus$'', and the characteristic polynomial behaves well with respect to this operation:
$\ch(T_1 \oplus T_2) = \ch(T_1) \ch(T_2)$  in $(1 + A[t])^\times$.  The characteristic polynomial then descends to a group homomorphism
\begin{equation}\label{eqn:main_map}
\begin{tikzcd}
\widetilde{K}_0 (\End(A)) \ar[r] & (1+A[[t]])^\times.
\end{tikzcd}
\end{equation}

The category of endomorphisms is endowed with another binary operation, the tensor product, $\otimes$. This turns $\widetilde{K}_0 (\End(A))$ into a ring. But we have already used up the natural multiplication on $(1+A[[t]])^\times$. How should we make \cref{eqn:main_map} a ring map?  The punchline is:
  \[
  \boxed{
    \begin{matrix}
      \textbf{Witt multiplication is exactly the operation}\\
      \textbf{that makes \cref{eqn:main_map} a ring map}.
      \end{matrix}
  }
  \]

This is due to Alkmvist \cite{almkvist1}. Witt multiplication is a device forced on us by linear algebra, and can similarly be extracted via linear algebra.

However, use of the characteristic polynomial is not necessarily the best way to prove relations in the Witt vectors. As we recall below, the characteristic polynomial is a repackaging of iterated traces. These turn out to produce the usual ghost maps. Furthermore, iterated traces exist in contexts where characteristic polynomials do not (for example, in the category of spectra).  A useful slogan is the following:
  \[
  \boxed{
    \begin{matrix}
      \textbf{All operations on Witt vectors are controlled by}\\
      \textbf{the combinatorics of iterated traces}
    \end{matrix}
    }
  \]

  This is still somewhat inadequate for a few reasons. First, as we note below, the packaging in terms of iterated traces has an issue when the underlying ring has torsion. Second, some Witt vectors cannot be captured by endomorphisms of finitely generated projective modules. Both problems can be repaired simultaneously by considering the Witt vectors as a completion of $\widetilde{K}_0 (\End(A))$:
  \[
  \boxed{
    \textbf{The Witt vectors are a completion of $\widetilde{K}_0 (\End(A))$}
  }
  \]

  This is the most satisfying explanation, encompassing all avatars of the Witt vectors, but as universal constructions usually are, it is the least useful for computations. However, the perspective taken leads to a presentation of Witt vectors which appears to be extremely useful for computation. This also leads to the view that:
\[
\boxed{
  \begin{matrix}
    \textbf{the Witt coordinates are a normal form}\\\textbf{for completed virtual endomorphisms}
  \end{matrix}
  }
\]
 in the same way that the rational canonical form is for matrices.

  When one takes the viewpoint that endomorphisms are the central object to be working with, the presence of the following structures also become more apparent.

\begin{itemize}
\item \textbf{Truncation sets} These arise from the combinatorics of iterated traces. If $\tr(f^i)$ is non-zero, then frequently, so is $\tr(f^{id})$.
\item \textbf{Symmetric polynomials} If one accepts traces of endomorphisms as the central object, then one way of studying them is via eigenvalues. Sums of iterated traces are then $\sum \lambda^n_i$, which can be profitably expressed in terms of symmetric polynomials. 
\item \textbf{$\lambda$-rings} The category of modules, and the category of modules of endomorphisms both have $\lambda$-operations. The $\lambda$-operations are continuous in the $V$-topology and so descend to operations on the Witt vectors. 
\item \textbf{The ring $(1+A[[t]])^\times$}: This is the natural target of the characteristic polynomial. 
\end{itemize}

Finally, if one is of a homotopy theoretic bent, the structure one witnesses as coming from $\widetilde{K}_0 (\End(A))$ should be behaving similarly for higher algebraic K-theory. Some conjectures suggest themselves, which we discuss more fully in \cref{sec:conjectures}. 

\begin{conj}
There is a filtration $V^\bullet$ on the algebraic $K$-theory spectrum $\widetilde{K}(\End(A))$ such that the first associated graded component is $\THH(A)$. The completion with respect to the filtration is $\TR(A)$. 
\end{conj}

We offer some reasonably compelling evidence for this conjecture, and  return to it in work-in-progress.

\subsection*{Outline}

In \cref{sec:witt} we review the classical definitions of the big Witt vectors and $p$-typical Witt vectors. 

In the next section, we explain a motivating example: the big Witt vectors of $\mathbb{Z}$ can be constructed as the $K$-theory of almost finite sets. 

In \cref{sec:main_section} we review the linear algebraic view of the Witt vectors. This is mostly a synthesis of \cite{almkvist1, grayson_witt}. The only new observation is the existence of the multiplicative norm arises at the categorical level. \cref{sec:p_typical} is a quick discussion of the $p$-typical Witt vectors from a linear algebraic perspective.

\cref{sec:verification} is devoted to using linear algebra to verify a number of identities in the Witt vectors. These verifications are considerably less painful than the usual ones. Again, much of this material is exposition of facts surely known to Almkvist and others. As a novel application, however,  we quickly recover identities involving Angeltveit's Witt norm. We discuss how the Artin-Hasse exponential can be viewed as a $p$-typical characteristic polynomial. 

\cref{sec:main_results} moves to a K-theoretic view of the Witt vectors. We show that $W(A)$ is a completion of the algebraic K-theory of endomorphisms. Following that, in \cref{sec:presentation}, we also give a new presentation of the Witt vectors, as well as a fairly algorithmic process for translating between ghost and Witt coordinates. 

Finally, \cref{sec:conjectures} contains some conjectures suggested by the analysis in this paper.

\subsection*{Acknowledgements}

The description of the rational Witt vectors in this paper has been around since Almkvist's papers, and the author was alerted to it by Dan Grayson's exposition \cite{grayson_witt}. The author has explained this perspective to many experts on Witt vectors, and it does not seem to be well known, hence this paper. 

David Mehrle provided helpful feedback on a draft of this paper, and caught many typos in the process. The author would like to thank Inna Zakharevich for encouraging him to write this paper, and for a careful and helpful reading of a draft. The author thanks Duke University and Kirsten Wickelgren for the opportunity to write it. 

\section{A Review of the Witt Vectors}\label{sec:witt}

We begin with the big Witt vectors. The definitions below are terse (but not much more terse than usual). For less idiosyncratic treatments of the Witt vectors, we recommend \cite{hesselholt_witt, hesselholt_madsen}. 

\begin{defn}\label{defn:witt_vectors}
  The \textbf{(big) Witt functor} is an endofunctor $W: \operatorname{Ring} \to \operatorname{Ring}$ with the following property. As a set,
  \[
  W(A) = \prod^\infty_{i=1} A = \{(a_1, a_2, \dots): a_i \in A\}
  \]
  The ring structure is determined by the requirement that the \textbf{ghost map}
  \[
  a_n \mapsto \sum_{d|n} d a_d^{n/d}
  \]
  determines a natural transformation $W(-) \to \prod(-)$ where the ring structure on the latter is given coordinate-wise. 
\end{defn}

A related definition is the $p$-typical version.

\begin{defn}
  The \textbf{$p$-typical Witt functor} is the endofunctor $W_{p^\infty}: \operatorname{Ring} \to \operatorname{Ring}$ with the following property. As a set
  \[
  W_{p^\infty} (A) = \prod^\infty_{i=0} A = \{(a_0, a_1, a_2, \dots): a_i \in A\}.
  \]
  The ring structure is determined by the requirement that the ``ghost map''
  \[
  a_n \mapsto \sum_i p^i a^{p^{n-i}}_i
  \]
  determines a natural transformation of functors $W_{p^\infty} (-) \to \prod(-)$
\end{defn}

\begin{rmk}
Note the minor difference in indexing conventions: we think of $W_{p^\infty}(A)$ as being indexed on $p^i$, hence starting with $i = 0$, which corresponds to 1,  instead of $i = 1$, which would correspond to $p$. 
\end{rmk}

From our perspective, there are a few problems with this definition. The first is that it is entirely opaque --- exactly why these ghost maps should arise is obscure. The second is that to define $p$-typical Witt vectors for $p$-torsion rings, one must invoke functoriality. The third is that it is not obvious that it is well defined. It is verified by proving the following.
\begin{thm}\label{thm:universal_polynomials}
  There are two families of polynomials (the \textbf{Witt universal polynomials})
  \begin{align*}
    s(x_1,x_2, \dots,  y_1, y_2, \dots), m(x_1,x_2, \dots, y_1, y_2,\dots) \in \mathbb{Z}[x_1,x_2, \dots, y_1, y_2, \dots]
  \end{align*}
  such that for $\mbf{a}, \mbf{b} \in W(A)$
  \begin{align*}
    \mbf{a}+\mbf{b} &= (s_1(\mbf{a},\mbf{b}), s_2(\mbf{a},\mbf{b}), \dots)\\
    \mbf{a}\cdot\mbf{b} &= (m_1 (\mbf{a},\mbf{b}), m_2(\mbf{a},\mbf{b}), \dots)
  \end{align*}
  An identical statement holds for $W_{p^\infty}(A)$. Furthermore, the polynomials $s_i, m_i$ depend only on $x_1, \dots, x_i, y_1, \dots, y_i$. 
\end{thm}

The (big) Witt vectors come equipped with two useful operations. Each of them must either be specified on the Witt coordinates, or else specified on the ghost coordinates, and shown to lift to the Witt coordinates. 

\begin{defn}
  The ring of big Witt vectors $W(A)$ is equipped with a family of additive maps $V^i: W(A) \to W(A)$ called the \textbf{Verschiebung(s)}. They are specified on Witt coordinates by
  \[
  V^i(a_1, a_2, \dots) = (0, \dots, \underbrace{a_1}_{ith\ \text{spot}}, \dots, \underbrace{a_2}_{\text{2ith}\ \text{spot}}, \dots)
  \]
\end{defn}

\begin{defn}
  The ring of big Witt vectors $W(A)$ is equipped with a family of multiplicative maps $F^i: W(A) \to W(A)$ called the \textbf{Frobenius maps} specified on ghost coordinates by
  \[
  F^i(a_1, a_2, \dots) = (a^i_1, a^i_2, \dots)
  \]
\end{defn}

There are similar definitions for the $p$-typical Witt vectors, but they involve a mild change in indexing.

\begin{defn}
The ring of $p$-typical Witt vectors $W_{p^\infty}(A)$ is equipped with an additive map $V: W_{p^\infty}(A) \to W_{p^\infty}(A)$ called the \textbf{Verschiebung}. It is specified in Witt coordinates by 
\[
V(a_0, a_1, \dots, ) = (0, a_0, a_1, \dots)
\]
\end{defn}

\begin{defn}
The ring of $p$-typical Witt vectors $W_{p^\infty}(A)$ is equipped with a multiplicative map $F: W(A) \to W(A)$ given on ghost coordinates by 
\[
F^p(a_0, a_1, \dots) = (a^p_0, a^p_1, \dots)
\]
\end{defn}

The standard way to check relations on the (big) Witt vectors is then to check on either the Witt or ghost coordinates. 

For later use, we note the following. The big Witt vectors are often encoded as the power series ring $(1+A[[t]])^\times$ with the identification
\[
(a_1, a_2, a_3, \dots) \mapsto \prod^\infty_{i=1} (1-a_i t^i)
\]
 As can be seen essentially from long division, every element in $(1+A[[t]])^\times$ can be written in the form of an element on the right. The group addition on $(1+A[[t]])^\times$ is given by power series multiplication. The multiplication, $\ast$, on $(1+A[[t]])^\times$ is defined by the requirement that $(1-at)\ast (1-bt) = (1-abt)$. Again, it must be shown that such a multiplication exists. 

There are topologies on $W(A)$ and $W_{p^\infty}(A)$ which will be important in the sequel. First, a definition. 

\begin{defn}
$W_n (A)$ will denote the \textbf{truncated big Witt vectors}. As a set, it is isomorphic to $\prod^n_{i=1} A$. It is given a ring structure by using the first $n$ universal Witt polynomials. 

Similarly, $W_{p^n} (A)$ denotes the \textbf{truncated $p$-typical Witt vectors}. As a set, it is isomorphic to $\prod^n_{i=1} A$. It is given a ring structure by the first $n$ universal $p$-typical Witt vectors. 
\end{defn}

The (big) Witt vectors are equipped with restriction maps $R: W_{n+1} (A) \to W_n (A)$  and $R:W_{p^{n+1}} (A) \to W_{p^n}(A)$ given by truncating the last element. It is then standard that
\[
W(A) = \varprojlim W_n (A) \qquad W_{p^\infty}(A) = \varprojlim W_{p^n}(A). 
\]
This presentation of $W(A)$ and $W_{p^\infty} (A)$ imbues them with topologies, which are typically called the \textbf{$V$-topology}. In the case of $W_{p^\infty} (A)$, this is because $W_{p^n} (A) = W_{p^\infty}(A)/V^nW_{p^\infty}(A)$ and
\[
W_{p^\infty} (A) \supset V^1 W_{p^\infty} (A) \supset V^2 W_{p^\infty} (A) \supset \cdots 
\]
is a decreasing filtration on $W(A)$. Similarly, if we let 
\[
V^{\geq i} W(A) = \{(0, \dots, 0, a_i, a_{i+1}, \dots) \in W(A)\}
\]
this gives a decreasing filtration on $W(A)$
\[
W(A) = V^{\geq 1} W(A) \supset V^{\geq 2} W(A) \supset \cdots 
\]
such that $W_n (A) = W(A)/V^{\geq n+1}W(A)$.  

\begin{rmk}
We chose the notation $V^i$ and $V^{\geq i}$ to emphasize the difference in the role that Verschiebung is playing. In the case of the $p$-typical Witt vectors, one really is filtering by iterated Verschiebungs, since there is only one $V$ map. This is not the case for the big Witt vectors. 
\end{rmk}

In the case of the big Witt vectors, the topology is also reflected on power series.

\begin{defn}
 Let $(1+A[[t]])^\times_{\geq n}$ denote the power series of the form $1 + a_n t^n + a_{n+1} t^{n+1}+\cdots$
\end{defn}

The following proposition captures everything we need about about the topology on power series.

\begin{prop}\label{lem:truncated_iso}
  $(1+A[[t]])^{\times}_{\geq n}$ is an ideal of $(1+A[[t]])^\times$ with Witt multiplication. The topologies specified by $V^{\geq n} W(A)$ and $(1+A[[t]])^\times_{\geq n}$ coincide, and
  \[
  \begin{tikzcd}
    W_n (A) \ar[r] & (1+A[[t]])^\times / (1+A[[t]])^\times_{\geq n}
  \end{tikzcd}
  \]
  is a ring homomorphism. 
\end{prop}

This finishes our exposition of the Witt vectors from a standard viewpoint.

\section{Almost Finite Sets}

As a prelude to the discussion linking $K$-theory and $W(A)$, this short section is a case study in the relationship between $K$-theory and the big Witt vectors of $\mathbb{Z}$. It has long been known that there are simple, $K$-theoretic, ways to present $W(\mathbb{Z})$. Such presentations come come up in \cite{dold_iterated, DS2}. In \cite{DS2} the authors consider a completed Burnside ring $\widehat{\Omega}_{C} (\mathbb{Z})$. In the case of \cite{dold_iterated}, the author considers a monoid he calls \textbf{PER}, which consists of infinite permutations, but of what he calls ``finite type'' and then group completes to $K_0 (\mbf{PER})$. In either case, the authors show they have a $K$-theoretic model for $W(\mathbb{Z})$.

These approaches are equivalent, and we use definitions from \cite{DS2}. 

\begin{defn}
Let the category of \textbf{almost finite sets}, $\Set^{\operatorname{af}}$ be the category of pairs $(X, \sigma)$ where $X$ is a countable set and  $\sigma: X \to X$ is a permutation such that the fixed point set of $\sigma^n$ is finite for each $n$. Morphisms $f:(X,\sigma_X) \to (Y, \sigma_Y)$ are bijections of sets $f: X \to Y$ that commute with the given permutations. The \textbf{Grothendieck group of almost finite sets} is $K_0 (\Set^{\text{af}})$. 
\end{defn}

\begin{rmk}
  The $K_0 (\Set^{\text{af}})$ is generated by formal sums of cycles with $\mathbb{Z}$-coefficients. That is, every element $[\sigma] \in K_0 (\Set^{\text{af}})$ can be written as a formal sum, with each $c_i$ finite,
  \[
  [\sigma]  = \sum^\infty_{i=1}  c_i C_i 
  \]
  where $C_i$ denotes an $i$-cycle. 
\end{rmk}

The key insight of \cite{DS2} is then the following. 

\begin{thm}\cite[Thm.~3]{DS2}
  Let $(1+Z[[t]])^\times$ be considered as a ring with Witt multiplication. Then map
  \[
  \begin{tikzcd}
    K_0 (\Set^\text{af}) \ar{r} & (1+\mathbb{Z}[[t]])^\times 
  \end{tikzcd}
  \qquad \sum^\infty_{i=1} c_i C_i \mapsto \prod^\infty_{i=1} (1-t^i)^{-c_i}
  \]
  is a ring isomorphism. 
\end{thm}

One of the virtues of this theorem is that if one were to take $K_0 (\Set^{\text{af}})$ as the \textit{definition} of the big Witt vectors of $\mathbb{Z}$, multiplication and addition would become fantastically simple. Elements are formal sums of cycles $C_i$, and from the combinatorics of cycles we have the following formulas (see \cite[Eqn~(2.5.4)]{DS2})
  \[
  C_i \cdot C_j \cong \operatorname{gcd}(i, j) C_{\operatorname{lcm}(i, j)}
  \]
  The Frobenius and Verschiebung operators can be given categorically. The Frobenius, $F^i$, (called ``restriction'' in \cite{DS2}) is given by applying each cycle $i$ times. The Verschiebung maps (called ``induction'' in \cite{DS2}) are slightly more complicated, but still defined at the level of the category of finite sets.

  It is also worth discussing the topology on $K_0 (\Set^{\text{af}})$. To the author's knowledge, Dress and Siebeneicher do not discuss this in \cite{DS2}, but it is useful for intuition. 

\begin{defn}
  Let $V^{\geq k} \Set^{\text{fin}}$ consist of finite sets and permutations whose cycle types are of size greater than $k$.
\end{defn}

The group $K_0 (V^{\geq k} \Set^{\text{fin}})$ is generated by \textit{finite} formal sums of cycles $C_j$ with $j \geq k$.

\begin{defn}
  Let $V^{\geq k} \Set^{\text{af}}$ consist of almost finite sets and permutations whose cycle types are of size greater than $k$.
\end{defn}

The group $K_0 (V^{\geq k} \Set^{\text{af}})$ is generated by \textit{almost finite} formal sums of cycles $C_j$ with $j \geq k$. That is,  expessions of the form $\sum^\infty_{i=k} n_i C_i$ where $n_i$ is finite. From the multiplicative structure given above, this is an ideal of the ring $K_0 (\Set^{\text{af}})$.

\begin{defn}
  The \textbf{$V$-topology} on $K_0 (\Set^{\text{af}})$ is given by the decreasing filtration by ideals
  \[
  K_0 (\Set^{\text{af}}) \supset  K_0 (V^{\geq 1} \Set^{\text{af}})  \supset K_0 (V^{\geq 2} \Set^{\text{af}}) \supset \cdots 
  \]
  There is a similar $V$-topology on $K_0 (\Set^{\text{fin}})$. We also define
  \[
  K_0 (\Set^{\text{af}})/V^n := K_0 (\Set^{\text{af}})/K_0 (V^{\geq n} \Set^{\text{af}}) 
  \]
  and similar for $K_0 (\Set^{\text{fin}}) / V^n$. 
\end{defn}

A consequence of these definitions is the following lemma. 

\begin{lem}
  The maps
  \[
  \begin{tikzcd}
    K_0 (\Set^{\text{fin}}) / V^n \ar{r}& K(\Set^{\text{af}})/V^n
  \end{tikzcd}
  \]
  and
  \[
  \begin{tikzcd}
    K_0 (\Set^{\text{af}}) \ar{r} &  \varprojlim K_0 (\Set^{\text{af}})/V^n
  \end{tikzcd}
  \]
  are ring isomorphisms. 
\end{lem}

And we obtain the following theorem. 

\begin{thm}
  Let $(1+\mathbb{Z}[[t]])^\times$ have the Witt ring structure, then
  \[
  \begin{tikzcd}
    \varprojlim K_0 (\Set^{\text{fin}})/V^n \ar{r} &  (1+\mathbb{Z}[[t]])^\times
  \end{tikzcd}
  \]
  is an isomorphism of rings. 
\end{thm}
\begin{proof}
  We know that $K_0 (\Set^{\text{af}}) \cong W(\mathbb{Z})$ is an isomorphism and that $K_0 (\Set^{\text{fin}})/V^n \cong K_0 (\Set^{\text{af}})/V^n$. Take the limit of both sides --- the only thing to worry about is $\varprojlim^1$ terms, but all the maps in the tower are surjective, so this vanishes. 
\end{proof}

This theorem is not so useful for computation, but it does have useful conceptual content: the Witt vectors arise from K-theory once we pass from a finite to infinite context in the most natural way possible.

\section{The Main Diagram}\label{sec:main_section}

We now come to the core of the paper. The viewpoint of this paper is that the ring of Witt vectors is best considered in terms of algebraic $K$-theory. To this end, we recall the definition of the $K$-theory of endomorphisms, and the key fact \cref{thm:workhorse}, due to Almkvist, which provides the impetus for this work. Roughly, it is the statement that the $K$-theory of the category of endomorphisms of finitely generated projective modules over a ring is isomorphic as a ring to the rational power series ring (here ``rational'' means fractions of polynomials, not defined over $\mathbb{Q}$). We then show that all the structure that one sees on Witt vectors can be recovered from this identification: Frobenius, Verschiebung, Norm, ghost maps, the $\lambda$-ring structure etc. Furthermore, they all have \textit{categorical} definitions. That is, they all come from functorial operations at the level of the category of endomorphisms. 

\begin{defn}\label{defn:k_end}
  Let $A$ be a commutative ring. Let $\operatorname{End}(A)$ denote the category of endomorphism of finitely generated projective $A$-modules. It has
  \begin{itemize}
  \item objects: Endomorphisms of finitely generated projective $A$-modules, $T: P \to P$
  \item morphisms: Given two endomorphisms $(P_1, T_1)$ and $(P_2, T_2)$, a morphism $(P_1, T_1) \to (P_2, T_2)$ is a map $f\colon P_1 \to P_2$ making the following diagram commute
    \[
    \begin{tikzcd}
      P_1\ar{d}{T_1}\ar{r}{f} & P_2 \ar{d}{T_2}\\
      P_1\ar{r}{f} & P_2
    \end{tikzcd}
    \]
  \end{itemize}
\end{defn}
The category $\End(A)$ is also an exact category: it has exact sequences given by sequences of endomorphisms that are exact upon forgetting to the category of modules. One can thus define a group completion $K_0 (\End(A))$ that has generators given by exact sequence $0 \to T_1 \to T_3 \to T_2 \to 0$ and relations $[T_3] = [T_2] + [T_1]$. There is a split surjection $K_0 (\End(A)) \to K_0 (A)$ given by forgetting the endomorphism, and the splitting is given by the zero endomorphism. 

\begin{defn}
  The \textbf{$K$-theory of endomorphisms} is the kernel of the split surjection \[
  K_0 (\End(A)) \to K_0 (A) \qquad (P, T) \mapsto P\]

  We denote it $\widetilde{K}_0 (\End(A))$. 
\end{defn}

We now state the theorem that is the backbone of this paper. 

\begin{thm}\cite{almkvist1}\cite[Cor.~3]{grayson_endo}\label{thm:workhorse}. 
  The characteristic polynomial
  \[
  \operatorname{ch}: \widetilde{K}_0 (\End(A)) \to (1+A[[t]])^\times 
  \]
  is injective, and its image is the rational power series. 
\end{thm}

The proofs in \cite{almkvist1} and \cite{grayson_endo} of \cref{thm:workhorse} are no more than a few pages long, and both are very clever. If one is only concerned with fields this is more or less a statement about the rational canonical form.

With this result in hand, we spend the rest of this section explaining the following diagram, which relates many standard constructions on Witt vectors:
\begin{equation}\label{big_diagram}
\begin{tikzcd}[row sep = 1.5cm]
  W_n (A) = A^{\times n}\ar{r}{\tau}\ar{dr}{\mathbf{gh}} \ar[bend left = 30]{rr} & \boxed{\widetilde{K}_0 (\End(A))} \ar{r}{ch}\ar{d}{\tr^{\bullet}} & (1 + A[[t]])^\times \ar{d}[swap]{t \frac{d}{dt} \log} \ar[bend right =50]{ll}\\
   & \prod A \ar{r}{\text{gen}}& A[[t]] \ar[bend right = 40]{u}
\end{tikzcd}
\end{equation}
A similar diagram, which also includes the ``necklace algebra'', appears in \cite[Eqn~1.7]{DS1}. 

Before explaining all of the objects and maps in the diagram, we note some structure on the central object, $\widetilde{K}_0 (\End(A))$.  Throughout, we will use $(P, f)$ to denote an endomorphism. All of the constructions below are \textit{categorical} constructions performed on the category $\End(A)$. The operations are extended to $\widetilde{K}_0 (\End(A))$ by additivity, unless otherwise noted.

First, $\End(A)$  is bimonoidal: it is equipped with binary operations  $\oplus$ and $\otimes$. These two operations give $\widetilde{K}_0 (\End(A))$ a ring structure.

There are also a number of functors on $\End(A)$ that descend to operators on $\widetilde{K}_0 (\End(A))$. 
\begin{defn}
$\End(A)$ is equipped with \textbf{Frobenius functors} $F^i: \End(A) \to \End(A)$. These are defined by
  \[
  F^i (P, f) \mapsto (P, f^i)
  \]
  where $f^n$ denotes $n$-fold composition.
\end{defn}

\begin{defn}
  $\End(A)$ is equipped with \textbf{Verschiebung functors} $V^i: \End(A) \to \End(A)$. These are defined by
  \[
  V^i (P, f) = (P^{\oplus i}, V^i(f))
  \]
  where we let $V^i(f)$ be the $(i \times i)$-matrix
  \[ V^i(f) =\begin{pmatrix}
    0 & 0 & \cdots & 0 & f\\
    1 & 0 & \cdots & 0 & 0 \\
    0 & 1 & \cdots & 0 & 0 \\
    \cdots & \ddots & \ddots & \vdots & \vdots \\
    0 & 0 & \cdots & 1 & 0
  \end{pmatrix}
  \]
\end{defn}

\begin{rmk}
  This can be pleasantly pictorially represented as a cyclically arrayed series
  
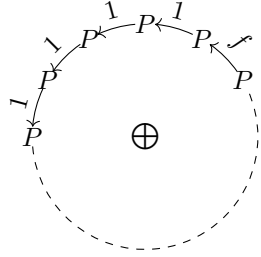
\[
\begin{tikzpicture}[scale=.75]
  \node[sloped] (c1) at (30:2cm) {$P$};
  \node (c2) at (60:2cm) {$P$};
  \node (c3) at (90:2cm)  {$P$};
  \node (c4) at (120:2cm) {$P$};
  \node (c5) at (150:2cm) {$P$};
  \node (c6) at (180:2cm) {$P$};
  \draw[->] (35:2cm) arc (35:55:2cm) node[midway, above, sloped] {$f$};
  \draw[->] (65:2cm) arc (65:85:2cm) node[midway, above, sloped] {$1$};
  \draw[->] (95:2cm) arc (95:115:2cm) node[midway, above, sloped] {$1$};
  \draw[->] (125:2cm) arc (125:145:2cm) node[midway, above, sloped] {$1$};
  \draw[->] (155:2cm) arc (155:175:2cm) node[midway, above, sloped] {$1$};
  \draw[dashed] (185:2cm) arc (185:385:2cm);
  \node (plus) at (0,0) {$\bigoplus$}; 
\end{tikzpicture}
\]
One should interpret all of the $P$s in this diagram as summed together. Verschiebung applies $f$ to one of them, and then rotates. 
\end{rmk}

\begin{defn}\label{defn:mult_verb}
The category $\End(A)$ is equipped with a \textbf{multiplicative Verschiebung} or \textbf{Norm functor} given by the composite
  \[
  \begin{tikzcd}
    \displaystyle \bigotimes^n_{i=1} P \ar{rr}{f \otimes 1 \cdots \otimes 1} & & \displaystyle \bigotimes^n_{i=1} P \ar{rr}{\operatorname{rot}}& & \displaystyle \bigotimes^{n}_{i=1} P
  \end{tikzcd}
  \]
  where ``rot'' means rotate one click counterclockwise.
\end{defn}

\begin{rmk}
As in the case of the additive Verschiebung, it is more clear to represent this diagrammatically:
    
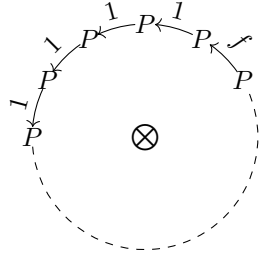
\[
\begin{tikzpicture}[scale=.75]
  \node[sloped] (c1) at (30:2cm) {$P$};
  \node (c2) at (60:2cm) {$P$};
  \node (c3) at (90:2cm)  {$P$};
  \node (c4) at (120:2cm) {$P$};
  \node (c5) at (150:2cm) {$P$};
  \node (c6) at (180:2cm) {$P$};
  \draw[->] (35:2cm) arc (35:55:2cm) node[midway, above, sloped] {$f$};
  \draw[->] (65:2cm) arc (65:85:2cm) node[midway, above, sloped] {$1$};
  \draw[->] (95:2cm) arc (95:115:2cm) node[midway, above, sloped] {$1$};
  \draw[->] (125:2cm) arc (125:145:2cm) node[midway, above, sloped] {$1$};
  \draw[->] (155:2cm) arc (155:175:2cm) node[midway, above, sloped] {$1$};
  \draw[dashed] (185:2cm) arc (185:385:2cm);
  \node (plus) at (0,0) {$\bigotimes$}; 
\end{tikzpicture}
\]
This does \textit{not} descend to a group homomorphism $N: \widetilde{K}_0 (\End(A)) \to \widetilde{K}_0 (\End(A))$, but it does extend to a multiplicative map. That is, it is a monoid map, respecting the operation given by multiplication. 
\end{rmk}

\begin{rmk}
Although it will not play much of a role in this paper, it is worth noting that the category $\End(A)$ is equipped with exterior power operations $(P, f) \mapsto (\Lambda^j P, \lambda^j f)$. These give $\widetilde{K}_0 (\End(A))$ the structure of a $\lambda$-ring.
\end{rmk}

As we will see later, these category-level Frobenius, Verschiebung and Norm govern the behavior of the usual Frobenius, Verschiebung and Norm. The $\lambda$-operations give the $\lambda$-ring structure on the Witt vectors. 

We now define the maps in the diagram. We begin with the central map. 

\begin{defn}
  The map
  \[
  \begin{tikzcd}
    \tr^\bullet : \widetilde{K}_0 (\End(A)) \ar{r} &\displaystyle \prod A
  \end{tikzcd}
  \]
  is given by
  \begin{align*}
    \tr^\bullet (f) &= (\tr(f), \tr(f^2), \tr(f^3), \dots)  \in \displaystyle \prod A \\
    &= (\tr(f), \tr(F^2 (f)), \tr(F^3 (f)), \dots)
  \end{align*}
  It is simply a map of the traces of iterates of $f$. 
\end{defn}

\begin{rmk}\label{mult_on_ghost}
Unlike many of the maps in the diagram, it is very easy to see that $\tr^\bullet$ is a ring map, if we define $\prod A$ to have pointwise addition and multiplication: the trace is both additive in direct sum and multiplicative in tensor product. This is one of the great virtues of giving this map a central role. 
\end{rmk}

The map $W_n (A) \to \End(A)$ is a composite of two maps. There is an obvious embedding $A \hookrightarrow \End(A)$ given by sending $a$ to the endomorphism given by multiplication by $a$:
\[
a : A \to A  = V^1(a) \in \End (A)
\]

We also have the following
\begin{defn}
  The \textbf{$n$th Teichm\"uller map}  $\tau_n : A \to \widetilde{K}_0 (\End(A))$ is given by
  \[
  \tau_n (a) = V^n (a) =  V^i(f) =\begin{pmatrix}
    0 & 0 & \cdots & 0 & a\\
    1 & 0 & \cdots & 0 & 0 \\
    0 & 1 & \cdots & 0 & 0 \\
    \cdots & \ddots & \ddots & \vdots & \vdots \\
    0 & 0 & \cdots & 1 & 0
  \end{pmatrix}
  \]
\end{defn}

\begin{defn}
  The \textbf{general Teichm\"{u}ller map} $\tau : W_n (A) = A^{\times n} \to \widetilde{K}_0 (\End(A))$ is given by
  \[
  \tau(a_1, \dots, a_n) = V^1 (a_1) \oplus V^2 (a_2) \oplus \cdots \oplus V^n (a_n)
  \]
\end{defn}

\begin{rmk}
This may seem a bit \textit{ad hoc}, but it is exactly this linear algebraic maneuver that yields all of the identities about Witt vectors. We will resolve some of the \textit{ad hoc} nature below. 
\end{rmk}

\begin{example}

  Considering the composite $\mathbf{gh}:= \tr^\bullet \circ \tau$ resolves some of the mystery around the usual definition of the ghost map.
  
  For example, consider what happens to $(a_1, a_2, a_3,a_4,a_5,a_6)$ under $\tr^\bullet \circ \tau$. Trace is additive on direct sum, so we consider each $V^i (a_i)$ separately
  \begin{align*}
    &\tr^\bullet (V_1(a_1)) &=  &a_1, &a^2_1,& &a^3_1,&  &a^4_1,&  &a^5_1 & &a^6_1&\dots\\
    &\tr^\bullet (V_2(a_2)) &=  &0  , &2a_2, & &0 ,  &  &2a^2_2,& &0 & & 2a^3_2&\dots\\
    &\tr^\bullet (V_3 (a_3)) &= &0,   &0,   &  &3a_3,&  &0,   &  &0 & & 3a^2_3&\dots\\
    &\tr^\bullet (V_4 (a_4)) &= &0,    &0,  &  &0,  &  &4a_4, &  &0 & & 0 &\dots\\
    &\tr^\bullet (V_5 (a_5)) &= &0,    &0,  &  &0,  &  &0     &  & 5a_5 & & 0 & \dots\\
    &\tr^\bullet (V_6 (a_6)) &= &0,    &0,  &  &0,  &  &0     &  & 0    & & 6a_6 & \dots\\
  \end{align*}
  Adding these vertically, using the component-wise sum on $\prod A$ we get
  \[
  \mathbf{gh} ((a_1, \dots, a_6)) = \left(a_1, a^2_1 + 2a_2, a^3_1 + 3a_3, a^4_1 + 2 a^2_2 + 4a_4, a^5_1 + 5a_5, a^6_1 + 2a^3_2 + 3a^2_3 + 6a_6, \dots  \right)
  \]
\end{example}

This gives a clear way of remembering the familiar definition of the ghost coordinates, which usually seem to come from nowhere
\[
\operatorname{gh}_n = \sum_{d|n} d a^{n/d}_d. 
\]

\begin{example}\label{example:ab_example}
  From such a table one sees that the ghost map is not necessarily surjective, and how to reconstruct a Witt vector from its coordinates. It even offers a way of doing it that is not much more difficult than long division, it is essentially ``$V$-division''

  Suppose we want to multiply $V^2(a) = (0,a)$ and $V^2(b)=(0,b)$ in Witt coordinates. In ghost coordinates, we see that their product is
  \begin{align*}
    0 && 4ab && 0 && 4a^2b^2&& 0 && 4a^3b^3 && 0 && 4a^4b^4.
  \end{align*}
  We want to find an element in $W(A)$ whose ghost coordinates are exactly this sequence. We begin the ``long division'' by finding $c$ such that $\Tr^2(V^2(c)) = 4ab$. But $\Tr^2 (V^2(c)) = 2c$, so $c = 2ab$. Now we compute iterated traces of $V^2(2ab)$ and subtract
    \begin{align*}
      &&  0 && 4ab && 0 && 4a^2b^2&& 0 && 4a^3b^3 && 0 && 4a^4b^4\\
      -  && 0 && 4ab  && 0 && 2 (2ab)^2 && 0 && 2 (2ab)^3 && 0 && 2(2ab)^4\\\hline
      && 0 && 0       && 0 && -4a^2b^2 && 0 && -12a^3b^3 && 0 && - 28a^4b^4  
    \end{align*}
    Continuing, we need $c$ such that $\tr^4(V^4 (c)) = 4c =  -4a^2 b^2$, so $c = -a^2 b^2$. And so on. If one continues in this way, one produces the necessary universal polynomials. This example is also worked out in \cite{grayson_witt} using the characteristic polynomial. 
\end{example}

Dwork's Lemma, which we state below, is also a consequence of these combinatorics. It is a formal statement of which elements the procedure above will work on (i.e. when we can evenly divide coefficients as we did above). 

\begin{prop}
  A sequence $(w_1, w_2, \dots)$ is in the image of the ghost map if
  \[
  w_{p^n a} \equiv w^{p}_{p^{n-1} a} \ \text{mod} \ p^n 
  \]
  for each prime $p$ and $\operatorname{gcd}(p, m) = 1$. 
\end{prop}

We continue with the rest of the maps in \cref{big_diagram}.

The top map is, of course, the characteristic polynomial. The ring structure on $(1 +A[[t]])^\times$ is assumed to have the structure that makes the characteristic polynomial a ring map, i.e. the Witt structure.

\begin{defn}
  The map $\ch: \widetilde{K}_0 (\End(A)) \to (1+A[[t]])^\times$ is given by a particular normalization of the characteristic polynomial
  \[
  \operatorname{ch}(P,f) = \det (\id - t f)
  \]
\end{defn}

\begin{rmk}
  There are 8 normalization of the characteristic polynomial:
  \[
  \det (\id \pm tf)^{\pm}, \  \det (t\id \pm f)^{\pm}
  \]
  These differences lurk in the background and make for confusing disparities between different treatments of Witt vectors. 
\end{rmk}

We have now defined two maps out of $\widetilde{K}_0 (\End(A))$: one defined by iterated traces, the other defined by the characteristic polynomial. The bottom horizontal map and the right vertical map relate them.  The bottom map in \cref{big_diagram} is simple to define. 

\begin{defn}
  The map $\mathbf{gen}: \prod A \to A[[t]]$ is the generating function map:
  \[
  (a_1, a_2, a_3, \dots) \mapsto \sum a_i t^i
  \]
\end{defn}

\begin{defn}
  The right vertical map is the \textbf{logarithmic derivative}
  \[
  p(t) \mapsto  t \frac{d}{dt} \log (p (t)) =  t \frac{p'(t)}{p(t)}
  \]
\end{defn}

This forces the definition of the upward map.

\begin{defn}
  The right vertical (upward) map is the inverse of the logarithmic derivative, i.e.
  \[
  g(t) \mapsto \exp\left(\int \frac{1}{t} g(t) \ dt\right)
  \]
\end{defn}

\begin{prop}
The diagram in \cref{big_diagram} commutes. 
\end{prop}
\begin{proof}
  To prove commutativity of the diagram, it suffices to check it on $V^i (a_i)$.

  First, the iterated traces of $V^i (a_i)$ are
  \[
  (\underbrace{i a_i}_{ith\ \text{spot}}, 0,\dots,0, \underbrace{i a^2_i}_{2ith \ \text{spot}},0, \dots, 0, \underbrace{i a^3_i}_{3ith \ \text{spot}}, \dots) 
  \]
Translating the iterated traces into generating functions, we obtain the series
  \[
  \sum i a^n_i t^{in}
  \]
  (recall that $a^n_i$ is in the $in$th spot, hence the coefficient on $t$).

  Going the other direction around the diagram, the characteristic polynomial of $V^i (a_i)$ is $(1 - a_i t^i)$
and the logarithmic derivative of the characteristic polynomial is
  \begin{align*}
    t\frac{d}{dt} \log (1-a_i t^i) &= t \frac{d}{dt} \log (1-a_i t^i)
    = t \left(\frac{-i a_i t^{i-1}}{(1-a_i t^i)}\right)
    = -ia_i t^i \left(\frac{1}{1-a_i t^i}\right)\\ &=-ia_i t^i \left(1 + a_i t^i + a^2_i t^{2i} + a^3_i t^{3i} + \cdots \right)
    = ia_i t^i + ia^2_i t^{2i} + ia^3_i t^{3i} + \cdots = \sum ia^n_i t^{in}
  \end{align*}
and we are done. 
\end{proof}

Summarizing, we have seen that the ghost maps come from considering traces of particular matrices. The structure of Witt multiplication are forced on us by a desire to have the characteristic polynomial be a ring map.

\section{The $p$-typical Witt vectors}\label{sec:p_typical}

The case of $p$-typical Witt vectors is no more difficult, and illuminated by thinking about the Witt vectors as recording traces. The $p$-typical Witt vectors are obtained by formally declaring that we only care about $p$-power traces of some endomorphism. To this end, we only use the Verschiebung operations $\{V^p, V^{p^2}, V^{p^3}, \dots\}$ and Frobenius $\{F^p, F^{p^2}, F^{p^3}, \dots\}$. We note that in the literature on $p$-typical Witt vectors these maps are referred to by their $p$-exponent, so $V_{p^i}$ is usually written $V_i$.

\begin{defn}
The \textbf{general Teichm\"{u}ller map} $ \tau: W_{p^n} (A) \to \widetilde{K}_0 (\End(A))$ is given by
\begin{equation}\label{teichmuller}
\tau(a_0, a_1, a_2, \dots,a_n) =  \oplus V^1 (a_0) \oplus V^p(a_1) \oplus V^{p^2}(a_2) \oplus \cdots V^{p^n} (a_n)
\end{equation}
\end{defn}

As is the case with the big Witt vectors, the ghost maps can be computed via traces. The difference is that we only care about traces whose order is a power of the prime.

\begin{example}
  Computing iterated traces in \cref{teichmuller}, we obtain
  \begin{align*}
    a_0 && a^p_0 && a^{p^2}_0 && a^{p^3}_0 && \cdots \\
    0   && pa_{1} && pa^{p}_1 && p a^{p^2}_1 &&\cdots \\
    0 && 0 && p^2 a_2 && p^2 a^{p}_2 && \cdots \\
    0 && 0 && \vdots && \vdots && \ddots
  \end{align*}
 and summing down columns, we get the ghost maps  which are
\[
gh_n(a_0, a_1, \dots) = \sum^n_{i=0} p^i a^{p-i}. 
\]
\end{example}

We can recover the universal polynomials again, by a relatively painless $V$-division algorithm. 

\begin{example}
  We compute the product of $(a_0, a_1)$ and $(b_0, b_1)$ to compute the first few universal polynomials. First, compute the ghost coordinates using the methods above, and then begin the $V$-division:
  \begin{align*}
    && a_0 b_0 && (a^p_0 + pa_1)(b^p_0 + pb_1) &&\\
    V^1(a_0b_0) && -a_0 b_0 && -a^p_0 b^p_0 && \\\hline
     && 0 && p a^p_0 b_1 + p a^p_1 b_0 + p^2 a_1 b_1 && \\
  \end{align*}
  So we need a $c$ such that $\tr^p (V^p(c)) = p c =  p a^p_0 b_1 + p a^p_1 b_0 + p^2 a_1 b_1$. Indeed, this is the universal Witt polynomial
  \[
   a^p_0 b_1 + a^p_1 b_0 + p a_1 b_1 
  \]
\end{example}

Here we get to a deficiency of this approach (and of all approaches based on ghost maps): the expressions involving traces do not make sense when there is $p$-torsion. This is in some sense the point of Witt vectors: they mimic the combinatorics of traces, without necessarily mentioning traces.

\begin{example}
  In the $p$-typical case something interesting happens for the characteristic polynomial. As far as the author knows, there is no ``characteristic polynomial'' defined linear algebraically which only extracts the data of $p$-power traces. Instead, working around \cref{big_diagram}, we see that the ``characteristic polynomial'' of $V^1(a)$ in this case is the Artin-Hasse exponential
  \[
  E(a) = \exp \left(at + \frac{a^pt^p}{p} + \frac{a^{p^2}t^{p^2}}{p^2} + \cdots \right)
  \]
To the author, at least, this explains the appearance of such a bizarre expression. If we do this with the other Verschiebungs, we get $E(V^{p^i}(a_i))$ is the same as $E(a_i)$ after performing the ring automorphism $t \mapsto t^{p^i}$. But this is the $i$th Verschiebung in the $p$-typical situation! Thus, from consideration of traces, we get a composite
\[
\begin{tikzcd}
\displaystyle\prod A \ar{r} & \widetilde{K}_0 (\End(A)) \ar{r} & (1+A[[t]])^\times 
\end{tikzcd}
\]
given by 
\[
(a_0, a_1, \dots,a_n) \mapsto V^{p^0}(a_0) \oplus V^{p^1} (a_1) \oplus \cdots \mapsto  \prod^n E(a_i X^{p^i}). 
\]
If we remove the central $\widetilde{K}_0 (\End(A))$ so that there are no finiteness assumptions, we obtain
\[
(a_0, a_1, \dots) \mapsto \prod E(a_i X^{p^i}) 
\]
which is a map whose image in the big Witt vectors is isomorphic to the $p$-typical Witt vectors. 

To reiterate, this was all obtained by just declaring, somewhat formally, that we only care about the traces of order $p^i$. 

\end{example}

\section{Identities, Old and New}\label{sec:verification}

In this section we present some of the standard identities on the big Witt vectors, and $p$-typical Witt vectors. An interesting consequence of the view we take is that the identities are often verified on a \textit{categorical} level, and one can always tell whether the operation is additive, multiplicative, etc. The typical procedure is to verify identities by showing that they agree on either all iterated traces, or on the characteristic polynomial; whichever is easier.  For this section, we will verify identities using traces. This requires us to assume a torsion-free (or $p$-torsion free) ring throughout this section. However, the usual Witt vector tricks work to port the identities: we invoke the functoriality of $\widetilde{K}_0 (\End(A))$  and the density of the rational Witt vectors in the Witt vectors to extend the results.  

\subsection{Old Identities}

As warm-up exercise in the use of the $K$-theory of endomorphisms, we prove some standard relationships between Frobenius and Verschiebung. None of these proofs are necessarily easier than the usual proofs that involve computing in ghost coordinates, but they are a bit more transparent. For discussion of the norm, the style of proof used below \textit{is} easier. 

\begin{prop}
In in $\widetilde{K}_0 (\End(A))$, $F^n V^n = n$
\end{prop}
\begin{proof}
  The operator $F^n$ simply raises a matrix to the $n$th power, and we've written down the Verschiebung matrix. 
  
  We compute
  \[
  F^n V^n = \begin{pmatrix}
    0 & 0 & \cdots & 0 & f\\
    1 & 0 & \cdots & 0 & 0 \\
    0 & 1 & \cdots & 0 & 0 \\
    \cdots & \ddots & \ddots & \vdots & \vdots \\
    0 & 0 & \cdots & 1 & 0    
  \end{pmatrix}^n =
  \begin{pmatrix}
    f & 0 & \cdots & 0 & 0\\
    0 & f & \cdots & 0 & 0 \\
    0 & 0 & \cdots & 0 & 0 \\
    \cdots & \ddots & \ddots & \vdots & \vdots \\
    0 & 0 & \cdots & 0 & f
  \end{pmatrix} = f^{\oplus n}
  \]
  and upon taking traces we get the result. 

\end{proof}

More precisley, we have the following lemma, which we will frequently use. 

\begin{lem}\label{lem:v_lemma}
  For $(V, f) \in \End(A)$
  \[
  \tr((V^d (f))^n) = \tr(F^n (V^d f) =
  \begin{cases}
    d \tr(f^{n/d}) & d | n \\
    0 & \text{otherwise}
  \end{cases}
  \]
\end{lem}
\begin{proof}
The endomorphism $F^n V^d (f)$ has 0 diagonal, unless $d |n$. In that case there are $d$ copies of the same endomorphism on the diagonal:
\[
\begin{pmatrix}
f^{n/d} & 0 & \cdots & 0 \\
0 & f^{n/d} & \cdots & 0 \\
\cdots & \cdots & \ddots & \vdots\\
0 & 0 &\cdots & f^{n/d} 
\end{pmatrix} = (f^{n/d})^{\oplus d} = d f^{n/d} 
\]
\end{proof}

The lemma generalizes to the following identity, which will be important in \cref{sec:presentation}. 

\begin{lem}\label{v_f_relation}
  Let $f \in \widetilde{K}_0(\End (A))$. Then
  \[
  F^m V^n (f) = \gcd(m,n) V^{n/(m,n)} F^{m/(m,n)} (f)
  \]
\end{lem}
\begin{proof}
  We check that this is true on all iterated traces. From the previous \cref{lem:v_lemma}
  \[ 
  \Tr^i (F^m V^n (f)) = \Tr ( (V^n (f))^{mi})  =
  \begin{cases}
    n \tr(f^{(mi)/n)}) &  n | mi \\
    0 & \text{otherwise} 
  \end{cases}
  \]
  and
  \begin{align*}
    &\operatorname{gcd} (m, n)\Tr^i(V^{n/\gcd(m,n)} (f^{m/\gcd(m,n)}))\\
    &=
  \begin{cases}
    \operatorname{gcd}(m,n) \left(\frac{n}{\operatorname{gcd}(m,n)}\right) \tr (f^{mi/n}) & \frac{n}{\operatorname{gcd}(m, n)} | \tfrac{m}{\gcd (m,n)}\\
    0 & \text{otherwise} 
  \end{cases}
  \end{align*}
  So we are done.
  
\end{proof}

\begin{lem}\label{other_v_f}
The maps $V^n, F^n$ in $\widetilde{K}_0 (\End(A))$ satisfy Frobenius reciprocity:
\[
f \otimes V^n (g) = V^n (F^n (f)\otimes  g)
\]
\end{lem}

\begin{proof}
   We want to show
  \[
  \tr^i (f \otimes V^n (g)) = \tr^i (V^n (F^n (f) \otimes g))
  \]
  for all $i$. We use formulas that we now have at our disposal, together with the multiplicativity of trace. By multiplicativity and \cref{lem:v_lemma}
\[
\tr^i(f \otimes V^n(g)) = \tr(f^i) \tr^i(V^n (g)) = 
\begin{cases}
n \tr(f^i) \tr(g^{i/n}) & n | i \\
0 & \ \text{otherwise}
\end{cases}
\]
Similarly, 

\[
\tr^i (V^n (F^n (f) \otimes g)) = 
\begin{cases} 
n \tr^i ( (F^n(f) \otimes g)^{i/n}) & n | i \\
0 & \text{otherwise}
\end{cases}
= 
\begin{cases}
n \tr(f^i) \tr(g^{i/n}) & n | i \\
0 & \text{otherwise}
\end{cases}.
\]
Thus, the operations coincide. 
\end{proof}

\subsection{Relations on the Norm}

In \cite{angeltveit} Angeltveit produces a multiplicative norm on Witt vectors. The construction proceeds in the usual way by requiring that certain identities hold on ghost coordinates and checking that they induce a multiplication in Witt coordinates. For the moment, we take our definition of ``Witt vectors'' to be $\widetilde{K}_0 (\End(A))$ and we examine what identities hold. 

\begin{defn}
  The \textbf{norm} on $\widetilde{K}_0 (\End(A))$ is defined via the multiplicative Verschiebung (\cref{defn:mult_verb}).  It is a multiplicative monoid homomorphism
  \[
  N^i: K_0 (\End(A))^{\otimes} \to  K_0 (\End(A))^{\otimes}
  \]
  that is, it does not respect the \textit{additive} structure. 
\end{defn}

Given this definition, the main theorem of \cite{angeltveit} admits a linear algebraic proof. 

\begin{prop}\cite[Thm.~1.2]{angeltveit}
Let $\mbf{a} \in W(A)$. Then $F^i \circ N^i (\mathbf{a}) = \mathbf{a}^i$
\end{prop}
\begin{proof}
  It is enough to prove for iterated traces of endomorphisms. First
  \[
  N^i (f) = \operatorname{rot}(\id \otimes \id \otimes \cdots \otimes f)
  \]
  That is, we take a $d$-fold tensor, apply $f$ to the last coordinate and rotate. Thus
  \[
  F^i ( N^i (f)) = N^i (f) \circ \dots \circ N^i (f) = \underbrace{f \otimes \cdots \otimes f}_{i \ \text{times}}
  \]
  By the multiplicativity of trace, we have   $\tr(F^i N^i (f)) = f^i$ and we are done.
\end{proof}

As in the case of the additive Verschiebung and Frobenius,  and the Norm and Frobenius interact more subtley depending on divisibility properties. Again, these divisibility properties come from the behavior of cyclic groups. 

\begin{example}
  Let $V$ be an $n$-dimensional vector space, with basis $e_1, \dots, e_n$. Then $V^{\otimes k}$ has basis $e_{i_1} \otimes \cdots \otimes e_{i_k}$. The norm map applies an endomorphism to the last coordinate and then rotates the basis elements. The fixed-points, and thus the trace, of this rotation action are clearly governed by divisibility properties. For example, $\operatorname{rot}^3$ fixes $e_1 \otimes e_2 \otimes e_3 \otimes e_1 \otimes e_2 \otimes e_3$. The basis elements $e^{\otimes k}_i$ are always fixed. Similarly, when computing traces, we only care about fixed points of basis element. Let $f$ be an endomorphism given by a matrix $(a_{ij})$. We examine what happens to a basis element under the norm
  \begin{align*}
    e_{i_1} \otimes \cdots \otimes e_{i_k} &\mapsto  e_{i_1} \otimes \cdots \otimes f(e_{i_k}) \\
    &= \sum_j e_{i_1} \otimes \cdots \otimes a_{i_k j} e_j \\
    &\mapsto \sum_j a_{i_k j} (e_j \otimes e_{i_1} \otimes \cdots \otimes e_{i_{k-1}})
  \end{align*}
  Thus, the terms where $j = i_1$, $i_1 = i_2$, $i_2 = i_3$, ..., contribute to the trace. Similar statements hold upon iteration. 
\end{example}

Given the above example, the following theorem is now easy.

\begin{thm}
  Let $p$ be a prime. Then
  \[
  \tr^i (N^{\otimes p} f) =
  \begin{cases}
    (\tr(f^{i/p}))^p & p \mid i \\
    \tr(f)^i & p \nmid i 
  \end{cases}
  \]
  Thus, for $d$ a positive integer
  \[
  \tr^i (N^{\otimes d} f) = \Tr(f^{i/\operatorname{gcd}(d,i)})^{\operatorname{gcd}(d,i)}
  \]
\end{thm}
\begin{proof}
We simply codify the previous example. 
\end{proof}

Combining the previous theorem, with fact that ghost coordinates are computed via traces, we obtain the following theorem of Angeltveit.

\begin{lem}\cite[Thm.~1.4]{angeltveit}
  On ghost coordinates, the norm maps $(x_t) \mapsto (y_t)$ are given by 
  \[
  y_t = x^g_{t/g} \qquad g = \operatorname{gcd}(d,t)
  \]
\end{lem}

\section{The $V$-Filtration}\label{sec:main_results}

In this section, we attempt to clarify more of the structure of the Witt vectors. From our perspective, there are a few issues with the Witt vectors as they stand. One is the kludgy feel of consistently reasoning with torsion-free rings and then porting answers to rings with torsion based on universal polynomials. Another is a lack of a clean answer for what the Witt vectors \textit{are}. A third is their completely uncomputable nature, despite their importance and concreteness. This is by no means an exhaustive list.  In this section we show that the Witt vectors are a completion of the ring $\widetilde{K}_0 (\End(A))$, which, in some sense, answers what the Witt vectors are. This definition will allow us to repair the rest of the issues. 

Before continuing, it is useful to have some discussion of the structure of the ring of rational power series $(1+A[[t]])^\times_{\text{rat}}$. We certainly have occasion to consider the rational power series, and also various truncations. 

\begin{notn}\label{notn:rational}
Let $(1+A[[t]])^\times_{\text{rat}}$ denote the rational power series, i.e. ones of the form
\[
\frac{1+a_1 t + a_2 t^2 + a_3 t^3 + \cdots + a_m t^m}{1 + b_1t + b_2 t^2 + b_3 t^3+ \cdots + b_n t^n}
\]
Let $(1+A[[t]])^\times_{\text{rat}, \geq n}$ denote the rational power series denote those of the form $1+a_n t^n + \cdots$. 
\end{notn}

By Almkvist's theorem, the characteristic polynomial map $\widetilde{K}_0 (\End(A))$ induces an \textit{isomorphism} of rings
\[
\begin{tikzcd}
\widetilde{K}_0 (\End(A)) \ar{r}{\cong} & (1+A[[t]])^\times_{\text{rat}}
\end{tikzcd}
\]
It is useful to know the inverse map. Suppose we have a rational power series written as in \cref{notn:rational}. Then, an explicit element in $\widetilde{K}_0 (\End(A))$ mapping to it is the virtual endomorphism below, obtained from companion matrices
\[
\left[\begin{pmatrix}
0 &0 &  0& \cdots & a_m\\
1 &0 & 0 &    &         a_{m-1}\\
 0 & 1 &   &    &         a_{m-2} \\
  &  &  \ddots &    &         \vdots \\
  &  &  &            1  & a_1
\end{pmatrix}, A^{\oplus m}\right]
-
\left[
\begin{pmatrix}
0 &0 &  0& \cdots & b_n\\
1 &0 & 0 &    &         b_{n-1}\\
 0 & 1 &   &    &         b_{n-2} \\
  &  &  \ddots &    &         \vdots \\
  &  &  &            1  & b_1
\end{pmatrix}, A^{\oplus n}
\right]
\]

We end our brief discussion with the following useful lemma

\begin{lem}\label{rational_filtered_equivalence}
There is an isomorphism of rings, induced by the inclusion map,
\[
\begin{tikzcd}
(1 + A[[t]])^{\times}_{\text{rat}} / (1+A[[t]])^{\times}_{\text{rat}, \geq n} \ar{r} &  (1+A[[t]])^\times / (1+A[[t]])^\times_{\geq n}
\end{tikzcd}
\]
\end{lem}
\begin{proof}
Both rings are power series $(1 + a_1 t+ \cdots + a_n t^n)$, with Witt multiplication and addition, modulo higher order terms. 
\end{proof}

The intuition behind the procedures in this section is that the characteristic polynomial and iterated traces encode the same data, and so one knows nearly complete information about an endomorphism from iterated traces. When taking iterated traces, passing from $\tr^i(f)$ to $\tr^{i+1}(f)$ moves new ring elements into the diagonal. The first $n$ Witt coordinate contain information, encoded in the most parsimonious fashion, about what new ring elements move in. Thus, we create a filtration based on traces. This also amounts to the same thing as creating a filtration based on Verschiebungs: as we have already seen, these control the ``depth'' of traces. 

\begin{defn}
The \textbf{$V$-filtration} $V^{\geq k} \widetilde{K}_0 (\End(A))$ is defined as follows.  The $k$th filtration level $V^{\geq k} \widetilde{K}_0 (\End(A))$ is the inverse image of $(1+A[[t]])^\times_{\geq k}$ under the characteristic polynomial map. 
\end{defn}

This suggests the following, more categorical, definition. 

\begin{defn}
The $V$-filtration $V^{\geq k}$ on $\End(A)$ is defined as follows.  Let $V^{\geq k} \End(A)$ be the full subcategory of $\End(A)$ consisting of those endomorphisms in the inverse image of $(1+A[t])^\times_{\geq n}$, the \textit{polynomials} of degree higher than $n$. 
\end{defn}

\begin{rmk}
We are trying to avoid introducing heavy language into this paper, but a better way to define this category is as the thick subcategory generated by the images of all $V^j$ with $j \geq k$. 
\end{rmk}

We have the following fact. 

\begin{lem}
The category $V^{\geq k} \End(A)$ is an exact category. 
\end{lem}
\begin{proof}
It is a full subcategory of an exact category, and it is closed under extension. 
\end{proof}

Thus, we can take $K$-theory $V^{\geq k} \End(A)$. It is useful for peace of mind to check the following. 

\begin{prop}
  The inclusion map induces an isomorphism of rings
  \[
  V^{\geq k} \widetilde{K}_0 (\End(A)) \cong \widetilde{K}_0 (V^{\geq k}\End(A))
  \]
\end{prop}
\begin{proof}
There is a clear inclusion map
\[
\begin{tikzcd}
\widetilde{K}_0 (V^{\geq k} \End(A))  \ar{r} & V^{\geq k} \widetilde{K}_0 (\End(A))
\end{tikzcd}
\]
We need only show that it is surjective. An element in $V^{\geq k} \widetilde{K}_0 (\End(A))$ can be written as a fraction of polynomials $p(x)/q(x)$. Each of $p(x), q(x)$ must be in $V^{\geq k} \widetilde{K}_0 (\End(A))$ as well. Lift these $p(x), q(x)$ to endomorphisms $T_p$ and $T_q$. By definition these endomorphisms are in $V^{\geq k} \End(A)$. Thus, we have lifted an element in $V^{\geq k} \widetilde{K}_0 (\End(A))$ to a virtual endomorphism $T_p - T_q$ and we are done. 
\end{proof} 

The $V$-filtration is a decreasing filtration:

\[
\End(A) \supset V^{\geq 1} \End (A) \supset V^{\geq 2} \End(A) \supset \cdots 
\]

and it descends to a filtration on $K_0$:
\[
\widetilde{K}_0 (\End(A)) \supset V^{\geq 1} \widetilde{K}_0 (\End(A)) \supset V^{\geq 2} \widetilde{K}_0 (\End(A)) \supset 
\]

Furthermore, the filtration behaves well with respect to the two binary operations. 

\begin{lem}
  Addition and multiplication respect the $V$-filtration on $\widetilde{K}_0 (\End(A))$ in the following sense
  \begin{align*}
    V^{\geq k_1} \widetilde{K}_0 (\End(A)) + V^{\geq k_2} \widetilde{K}_0 (\End(A)) \subset V^{\min\{k_1,k_2\}} \widetilde{K}_0 (\End(A)) \\
    V^{\geq k_1} \widetilde{K}_0 (\End(A)) V^{\geq k_2} \widetilde{K}_0 (\End(A)) \subset V^{\operatorname{lcm} (k_1, k_2)} \widetilde{K}_0 (\End(A)) 
  \end{align*}
\end{lem}
\begin{proof}
Follows from the additive and multiplicative properties of traces. 
\end{proof}

\begin{rmk}
The filtration is not a multiplicative filtration since that usually requires $F^m \cdot F^n \subset F^{m+n}$. 
\end{rmk}

Now we consider quotients in the filtration.  We abbreviate the quotients $\widetilde{K}_0 (\End(A)) / V^{\geq j} \widetilde{K}_0 (\End(A))$  to $\widetilde{K}_0 (\End(A))/V^j$.  The filtration induces maps between these quotients, and thus a tower 
\[
\widetilde{K}_0 (\End(A))/V \leftarrow \widetilde{K}_0 (\End(A))/V^2 \leftarrow \widetilde{K}_0 (\End(A))/V^2 \leftarrow 
\]

This leads to the central definition. 

\begin{defn}
  The \textbf{$V$-completion} of $\widetilde{K}_0 (\End(A))$ is the limit
  \[
  \varprojlim_k \widetilde{K}_0 (\End(A))/V^j
  \]
\end{defn}

\begin{thm}
  The generalized Teichm\"{u}ller map induces an isomorphism of (topological) rings
  \[
  \begin{tikzcd}W(A) \ar{r}{\cong} & \varprojlim_k \widetilde{K}_0 (\End(A))/V^j\end{tikzcd}
  \]
\end{thm}
\begin{proof}

  We work with the truncations.
  
  First, the characteristic polynomial descends to an isomorphism of rings
  \[
\begin{tikzcd}
\widetilde{K}_0 (\End(A))/V^n \ar{r}{\cong} & (1+A[[t]])^{\times}_{rat}/(1+A[[t]])^{\times}_{\geq n, rat}
\end{tikzcd}
\]
To see this, we have sequences of ideals and quotients:
\[
\begin{tikzcd}
V^{\geq n} \widetilde{K}_0 (\End(A)) \ar{r}\ar{d}{\cong} & \widetilde{K}_0 (\End(A)) \ar{r}\ar{d}{\cong} & \widetilde{K}_0 (\End(A))/V^n\ar{d} \\
(1+A[[t]])^{\times}_{\geq n , rat}\ar{r} & (1 + A[[t]])^{\times}_{rat}\ar{r} & (1+A[[t]])^{\times}_{\leq n, rat}
\end{tikzcd}
\]

In the diagram above, the first isomorphism is by definition, and the second is Almkvist's theorem.

Second, the generalized Teichm\"{u}ller map descends to an isomorphism of rings. Recall that this is the map $W_n (A) \to \widetilde{K}_0 (\End(A))/V^n$ given by  $(a_1, \dots, a_n) \mapsto \bigoplus^n_{i=1} V^i(a_i)$
followed by the quotient map. To see it is an isomorphism of rings we consider the composition
\[
\begin{tikzcd}
W_n (A) \ar{r} & \widetilde{K}_0 (\End(A))/V^n \ar{r}{\cong} & (1+A[[t]])^{\times}_{rat}/(1+A[[t]])^{\times}_{\geq n, rat}
\end{tikzcd}
\]
As we have just shown, the second map is an isomorphism of rings, and the composite is an isomorphism by \cref{lem:truncated_iso} and \cref{rational_filtered_equivalence}. Thus $W_n (A) \cong \widetilde{K}_0 (\End(A))/V^n$. 

Finally, we consider the diagram
\[
\begin{tikzcd}
W_{n+1} (A) \ar{d} \ar{r}{\cong} & \widetilde{K}_0 (\End(A))/V^{n+1} \ar{d}\\
W_n (A) \ar{r}{\cong} & \widetilde{K}_0 (\End(A))/V^n
\end{tikzcd}
\]
It is not hard to see that the diagram commutes. Upon taking limits, and noting the vanishing of $\varprojlim^1$ terms by surjectivity, we get the result. 
\end{proof}

In general, we thus suggest the following definition for Witt vectors. 

\begin{defn}
Let $A$ be a ring (which may or may not be commutative). Then 
\[
\overline{W}(A) := \varprojlim_j\widetilde{K}_0 (\End(A))/V^j
\]
\end{defn}

\begin{rmk}
We have shown that for a commutative rings $\overline{W}(A) \cong W(A)$. In future work, we hope to show that the definition also coincides for non-commutative rings. It should be true, and should go through the identification of topological restriction homology, $\TR$ and Hesselholt's identification of the non-commutative Witt vectors of $A$ with $\TR(A)$ \cite[Thm.~2.2.9]{hesselholt_noncommutative}.  
\end{rmk}

There are a few advantages to this definition. One is it allows us to define operations on Witt vectors in a categorical way.

\begin{prop}
Let $G$ be an endofunctor $G: \End(A) \to \End(A)$ such that the induced map $G: \widetilde{K}_0(\End(A)) \to \widetilde{K}_0(\End(A))$ is continuous in the $V$-topology. Then $G$ descends to $\overline{W}(A)$. 
\end{prop}

\begin{cor}
There is a (multiplicative) norm map on $\overline{W}(A)$ given by the multiplicative Verschiebung.
\end{cor}

Another advantage is that the Witt vectors only depend on the underlying category of finitely generated projective modules. Then, accepting this as the definition, we have the following theorem. 

\begin{thm}
The Witt vectors are Morita invariant: if $A, B$ are Morita equivalent rings, then $\overline{W}(A) \cong \overline{W}(B)$. 
\end{thm}

As stated, this is not so useful. The canonical example of Morita equivalent rings are $A$ and the matrix ring $M_n (A)$, and the latter is noncommutative. We have not checked that this definition of Witt vectors agrees with the definition of non-commutative Witt vectors. However, as mentioned, this should hold. 

The Witt vectors do not have any pleasant universal property as far as we know --- and this is perhaps one illustration why: they are a limit built out of colimits built out of a group completion. In order to state a nice universal property, the best way is then to pass to some category already equipped with the structure we see on $\widetilde{K}_0 (\End(A))$. For example, one would want to have both a Frobenius and a Verschiebung which should satisfy certain identities. This perhaps motivates the appearance of pro-$V$-complexes or Cartier modules in much work  that has to do with $K$-theory (see, e.g. \cite{antieau_nikolaus}).

\section{A Presentation}\label{sec:presentation}

In this section we give a presentation of $W(A) \cong \overline{W}(A)$. The idea is that instead of working with algebraic shadows of endomorphisms, we should work with the endomorphisms directly.  Computations in this presentation are very tractable. Addition is an obvious operation, and multiplication is given by an extremely simple formula. No universal polynomials are required, and no torsion assumptions are required --- we are working with \textit{formal} sums, and so the multiplication by elements in $\mathbb{Z}$ never interferes with the ring elements. This is in some sense the source of all the headaches with torsion in the standard treatments of Witt vectors: upon taking traces, these \textit{formal} coefficients begin interacting with the ring elements.

 One of the reasons that Witt arithmetic is so complicated is that it involves two steps: one is the actual arithmetic, and other is converting the answer into a Teichm\"{u}ller representative (see \cref{defn:teichmuller_representative}), or a kind of normal form. By separating out the two steps, each becomes much easier.

To begin, each element of $\overline{W}(A)$ can be represented by a formal sum of Verschiebungs

\begin{align*}
  &c_{11} V^1(a_{11}) \oplus c_{12} V^1 (a_{12}) \oplus \cdots \oplus c_{1n_1} V^1(a_{1n_1}) \oplus\\
  &c_{21} V^2 (a_{21}) \oplus c_{22}  V^2(a_{22}) \oplus \cdots \oplus c_{2n_2} V^2(a_{2n_2})\oplus\\
  &c_{31} V^3 (a_{31}) \oplus c_{32}  V^3(a_{32}) \oplus \cdots \oplus c_{3n_3} V^3(a_{3n_3})\oplus\\
  &\vdots
\end{align*}

as can be seen via the characteristic isomorphism $\operatorname{ch}: \overline{W}(A) \to (1+A[[t]])^\times$.

\begin{rmk}
  We think of these Verschiebungs as the atomic bits of endomorphisms. Adding $V^n(a_n)$ adds some extra iterated traces to the ghost map, or an extra term to the characteristic polynomial, past degree $n$. 
\end{rmk}

Of course, we have the usual Verschiebung and Frobenius operators. We also define a trace operator. 

\begin{defn}
The Verschiebung map is given by the formula
\[
  V^k \left(\sum c_i V^i(a_i)\right) = \sum c_i V^{ik}(a_i)
  \]
\end{defn}

\begin{defn}
  The Frobenius map is given by the formula 
  \[
F^k \left(\sum c_i V^i (a_i) \right) = \sum c_i \operatorname{gcd}(k,i) V^{i/\operatorname{gcd}(k,i)} (a^{k/\operatorname{gcd}(k,i)}_i)
\]
\end{defn}
That is, the action of Frobenius is given by \cref{v_f_relation}. This relation can also be proved using the characteristic polynomial. 

\begin{defn}
The \textbf{trace}, $\tr: \overline{W}(A) \to A$ is given by
\[
\tr\left(\sum c_i V^i(a_i)\right)= c_1 a_1
\]
Note that the trace is a ring homomorphism. It is also important to observe that $c_1 a_1$ is an actual ring element, rather than a formal endomorphism. 
\end{defn}

\begin{rmk}
Though it is only an analogy or curiosity, it is worth mentioning that the trace is really a residue --- it extracts a single well-defined coordinate in a formal series. If one consider the characteristic polynomial with the inverse normalization $\det (\id - tf)^{-1}$, then trace is a residue in the sense of complex analysis. 
\end{rmk}

To avoid use of the ghost map, we redevelop the necessary relations using only the characteristic polynomial.

\begin{prop}\label{prop:necessary_relations}
With definitions as above, the following identities hold
\begin{itemize}
\item $V^1(a) V^i (b) = V^i (F^i(V^1(a)) V^1(b))$ for $a, b$ ring elements in $A$
\item $F^i V^j (a) = \gcd(i,j) V^{i/\gcd(i,j)} F^{j/gcd(i,j)}(a)$ 
\end{itemize}
\end{prop}

\begin{rmk}
  The appearance of the $V^1$ terms may seem odd, but we want to distinguish between elements of $\overline{W}(A)$ (i.e. an element in the ring of completed endomorphisms) and a ring element $a \in A$. These two objects behave quite differently. For instance, for $n \in \mathbb{Z}$, the morphisms $nV^1(a)$ and $V^1(na)$ are
  \[
  nV^1(a) =
  \begin{pmatrix}
    a & &  \\
    & \ddots & \\
     & & a  
 \end{pmatrix}
  \qquad
  V^1(na) = (na). 
  \]
  The elision of these two objects is somewhat dangerous. 
\end{rmk}

\begin{proof}[Proof of Proposition]
We check each identity on characteristic polynomials. The matrices for $V^1(a) V^i (b)$ and $V^i (F^i(V^1(a)) V^1(b))$ are
\[
\begin{pmatrix}
0 & \cdots & &0& ab \\
a & \cdots & &0& 0 \\
0 & a & \cdots & 0 &0 \\
  &   &        &   & \\
0 & 0 & \cdots & a & 0\\
\end{pmatrix}
\qquad
\begin{pmatrix}
0 & \cdots & &0& a^ib \\
1 & \cdots & &0& 0 \\
0 & 1 & \cdots & 0 &0 \\
  &   &        &   & \\
0 & 0 & \cdots & 1 & 0\\
\end{pmatrix}
\]
and thus the corresponding characteristic polynomials are computed as the determinants
\[
\begin{pmatrix}
1 & \cdots &\cdots &0& abt \\
-ta & 1 &\cdots &0& 0 \\
0 & -ta & \cdots & 0 &0 \\
  &   &        &   & \\
0 & 0 & \cdots & -ta & 1\\
\end{pmatrix}
\qquad
\begin{pmatrix}
1 & \cdots & &0& -a^ibt \\
-t & 1 & &0& 0 \\
0 & -t & \cdots & 0 &0 \\
  &   &        &   & \\
0 & 0 & \cdots & -t & 0\\
\end{pmatrix}
\]
Using the usual ``cyclic trick'' for computing determinants, these obviously have the same determinant. 

For the second identity it is enough to show that for primes $F^p V^p = p V^1$ and $F^p V^q = V^q F^p$ for distinct primes, $p$, $q$.

The identity $F^p V^p = p V^1$ is true at the level of endomorphisms.

To see $F^p V^q = V^q F^p$ for distinct primes, we simply write the determinants that would verify this:
\[
\begin{pmatrix}
1 & \cdots & &0& a \\
-at & 1 & &0& 0 \\
0 & -at & \cdots & 0 &0 \\
  &   &        &   & \\
0 & 0 & \cdots & -t & 1\\
\end{pmatrix}
\qquad
\begin{pmatrix}
1 & \cdots & &0& a^p \\
-t & 1 & &0& 0 \\
0 & -t & \cdots & 0 &0 \\
  &   &        &   & \\
0 & 0 & \cdots & -t & 1\\
\end{pmatrix}
\]
where in the left matrix there are $p$ copies of $-at$ down the diagonal. Again, their determinants are equal. 

\end{proof}

Finally, we get the following, without recourse to ghost maps. 

\begin{thm}\label{thm:presentation}
  As an abelian group $\overline{W}(A)$ is formal sums of $V^i(a_i)$ with $\mathbb{Z}$-coefficients.  It acquires a ring structure when multiplication is given by 
  \[
  V^i (a_i) V^j (a_j) = \operatorname{gcd} (i, j) V^{\operatorname{lcm}(i,j)} (a^{j/\gcd(i,j)}_i  a^{i/\gcd(i,j)}_j) 
  \]
\end{thm}

\begin{proof}
  As a group, $\overline{W}(A)$ has this presentation. One must only check that the relations on multiplication hold. However, this is a consequence of \cref{prop:necessary_relations}
We compute
  \begin{align*}
    V^i (a_i) V^j (a_j) &= V^j (F^j (V^i (a_i)) V^1(a_j))\\
    &=V^j (\gcd (i, j) V^{i/\gcd(i,j)} F^{j/\gcd(i,j)} V^1(a_i)) V^1(a_j))\\
    &= \gcd(i,j) V^j ((V^{i/\gcd(i,j)} a^{j/(i,j)}_i) V^1(a_j))\\
    &= \gcd(i,j) V^j (V^{i/\gcd(i,j)} (V^1(a^{j/(i,j)}_i)  F^{i/\gcd(i,j)} V^1(a_j))) \\
    &= \gcd (i,j) V^j (V^{i/\gcd(i,j)} (a^{j/\gcd(i,j)}_i a^{i/\gcd(i,j)}_j)\\
    &=  \gcd(i,j) V^{\operatorname{lcm}(i,j)}(a^{j/\gcd(i,j)}_i a^{i/\gcd(i,j)}_j)
  \end{align*}
\end{proof}

\begin{example}
  We return to the example of $(0,a)\ast (0,b)$, which is $V_2 (a) V_2(b)$. By the presentation, we get $2V_2 (ab)$. Indeed, one can then check that under the characteristic polynomial map, this goes to $1-2ab t^2 + a^2 b^2 t^4$, as Grayson gets in \cite{grayson_witt}.  This is even easier than multiplication using endomorphisms and the characteristic polynomial. 
\end{example}

Some remarks are in order.

\begin{rmk}
Identities in this form are well known. However, the derivation of the corresponding identity relies on the same kind of unmotivated trick of desiring a functorial multiplication $\ast$ on $(1+A[[t]])^\times$ such that $(1-at)\ast (1-bt) = (1-abt)$ and then showing that such a multiplication exists. In the above, a multiplication \textit{does} exist on $\overline{W}(A)$ because $\operatorname{End}(A)$ has a tensor monoidal structure. Working out the multiplication only depended on knowing that the characteristic polynomial is injective, so can detect equivalences in $\overline{W}(A)$. Thus, the multiplicative structure is a clear \textit{consequence} rather than a requirement. 
\end{rmk}

\begin{rmk}
This presentation appears to be quite similar to the approach of Cuntz--Deninger \cite{cuntz_deninger}. They essentially work formally in the ghost coordinates (i.e. work on traces). It seems in their approach that the group ring $\mathbb{Z}[A]$ is used to keep track of the difference between $n V^1(a)$ and $V^1(na)$. 
\end{rmk}

\begin{rmk}
We began this paper arguing that Witt vectors are trying to encode of the combinatorics of traces. Interpreting $V^i(a_i)$ as an atomic element of trace, this is exactly what this presentation of $\overline{W}(A)$ is doing. 
\end{rmk}

The presentation above is quite satisfying, but one may still want to express a formal sum $\sum c_{ij} V^i(a_{ij})$ with $c_{ij} \in \mathbb{Z}$ as a sum in the image of the Teichm\"{u}ller map, i.e. of the form $\sum V^i(a_i)$. 

Since this is a useful concept, we make the following definition.

\begin{defn}\label{defn:teichmuller_representative}
  Let $\mbf{a} \in W(A)$. A \textbf{Teichm\"{u}ller representative} for $\mbf{a}$ is an element $\overline{\mbf{a}} \in W(A)$ written in the form $\bigoplus_{i=1} V^i(\overline{a}_i)$ such that $\mbf{a} = \overline{\mbf{a}}$.
\end{defn}

\begin{rmk}
  We sometimes say $\mbf{\overline{a}}$ is in \textbf{Teichm\"{u}ller form} to emphasize the formal similarity with putting a matrix in a canonical form: it amounts to choosing a particular basis to represent endomorphisms. 
\end{rmk}

The following example is useful for intuition:

\begin{example}\label{weird_procedure}
  For this example assume that the ring we are working in is torsion-free. One could imagine trying to build an endomorphism up from what $k$th traces are. \textit{This is exactly what the Witt coordinates do}. Suppose, for example, that $\tr^k (f) = a_k$ and $\tr^i(f) =0$ for $i \leq k$. Then
  \[
  f - \frac{1}{k}V^k(a_k) \in V^{\geq k+1}.  
  \]
  This is really the $V$-division procedure. One could repeat this:
  \begin{align*}
    f - V(\tr(f)) &\in V^{\geq 2} \\
    f - V_2 \left(\frac{1}{2}\tr^2(f - V(\tr(f)))\right) &\in V^{\geq 3}\\
    f - V_3 \left(\frac{1}{3} \tr^3\left(f - \frac{1}{2} V_2 (f - V(\tr(f)))\right)\right) &\in V^{\geq 4} \\
    &\vdots 
  \end{align*}
  The Witt coordinates are then what is subtracted off at each stage. 
\end{example}

If the above example seems familiar, it is because it is formally the same as writing down the Newton polynomials. That is, the Witt coordinates and the ghost coordinates have the same relationships to each other as the symmetric polynomials and power polynomials. Indeed, if we imagine the the endomorphism has having eigenvalues $\lambda_i$, then the iterated traces are $\sum \lambda^k_i$ and the Witt coordinates should be the elementary symmetric polynomials in this coordinates (up to a sign). 

If we choose the normalization of the elementary symmetric polynomials where
\[
\sigma_k = (-1)^k \sum x_{i_1} \cdots x_{i_k} \qquad \sigma_0 = 1
\]
then, we have the identities
\[
-k\sigma^n_k = (-1)^{k-1}\sigma_{k-1} p_1 + (-1)^{k-2}\sigma_{k-2} p_2 + \cdots +  (-1)^1 \sigma_1 p_{k-1} - p_k
\]
which allow for recursive computation of the $\sigma_k$ in terms of $p_k$.

We have up until this point avoided ghost coordinates, but given that we want to translate into Witt coordinates, we can no longer avoid them. However, with the correct definitions, we avoid issues with torsion.

\begin{defn}
  The \textbf{modified trace map}, $\overline{\tr}: \overline{W}(A) \to V^{\geq 1}\overline{W}(A) / V^{\geq 2} \overline{W}(A)$ is given by
  \[
  \begin{tikzcd}
    c_{11} V^1(a_{11}) + c_{12} V^1(a_{12}) + \cdots + c_{1n} V^1(a_{1n}) + \underbrace{\text{higher order terms}}_{\in V^{\geq 2} \overline{W}(A)}\ar[mapsto, d] \\
    c_{11} V^1(a_{11}) + c_{12} V^1(a_{12}) + \cdots + c_{1n} V^1(a_{1n}) 
  \end{tikzcd}
  \]
\end{defn}

This is again some kind of residue. We are not going all the way down to the ring $A$, but instead considering the trace as recording some core, well-defined, term in a power series. 

By the properties we have proved for the multiplicative structure, we immediately have the following

\begin{lem}
The modified trace is a ring homomorphism
\end{lem}

The following lemma is also useful.

\begin{lem}
  In $V^{\geq 1} \overline{W}(A)/V^{\geq 2} \overline{W}(A)$, $V^1(a+b) = V^1(a)+V^1(b)$. Consequently, $V^1(na) = n V^1(a)$ for $n \in \mathbb{Z}$. 
\end{lem}
\begin{proof}
This is an easy check with characteristic polynomials. 
\end{proof}

For ease, we abbreviate the target of the modified ghost map to $V^{\geq 1}/V^{\geq 2}$, with the ring understood. Since we now have a trace, and we already had a Frobenius map, we can define a version of the ghost map. 

\begin{defn}
  Define the \textbf{modified ghost map} on 
  \[
  \mbf{x} \mapsto (\overline{\tr} \circ F^n) (\mbf{x}) \qquad \overline{W}(A) \to \prod V^{\geq 1} / V^{\geq 2}
  \]
  The term $\prod V^{\geq 1}/V^{\geq 2}$ is given a ring structure with pointwise multiplication, and is called the \textbf{modified ghost ring}. 
\end{defn}

Given the modified ghost coordinates, and computing only in the modified ghost ring, we can use Newton's relations formally to compute the appropriate Witt coordinates. 

\begin{example}
  We consider $(0,a)\ast(0,b) = 2V^2(ab)$. We compute the modified ghost coordinates
  \[
  0 \qquad 4V^1(ab) \qquad 0 \qquad 4 V^1(a^2b^2) \qquad 0 \qquad 4 V^1(a^3 b^3) \qquad 0 \qquad 4 V^1(a^4b^4)
  \]
  First, $-\sigma_1 = p_1 = 0$. 
  
  The second term is computed as follows: $-2 \sigma_2 = -\sigma_1 p_1 - p_2$, and $p_2$ is  $2V^1(ab) = V^1(2ab)$ (in $V^{\geq 1}/V^{\geq 2}$!)

  Continuing, we compute 
  \begin{align*}
    -4 \sigma_4 &= -\sigma_3 p_1 + \sigma_2 p_2 + -\sigma_1 p_3 - p_4\\
    &= 0 + 2V^1(ab) 4V^1(ab) - 4V^1(ab)V^1(ab)\\
    &= 4V^1(a^2 b^2)
  \end{align*}
  so $\sigma_4 = V^1(-(ab)^2)$, etc.

  The Teichm\"{u}ller normal form is then given by
  \[
  V^1(\overline{\sigma_1}) + V^2(\overline{\sigma_2}) + V^3(\overline{\sigma_3}) + V^4(\overline{\sigma_4}) + \cdots
  \]
   where $\overline{\sigma}_i$ denotes the lift of $\sigma_i \in V^{\geq 1}/V^{\geq 2}$ to $\overline{W}(A)$. In our current example, the sum becomes
  \[
  V^2 (V^1(2ab)) + V^4 (V^1(-(ab)^2)) + \cdots = V^2(2ab) + V^4(-(ab)^2) + \cdots 
  \]
\end{example}

So, by use of the correct definitions, we have avoided ever having to mention ghost coordinates in the definition of the ring of Witt vectors, or to prove any of their basic properties and identities. Indeed, that would be akin to diagonalizing a matrix, or putting it into a normal form, to prove all matrix identities. However, in some cases a normal form is convenient, for instance, in testing for equality. To this end, we have given a procedure for translating into Witt coordinates, which, again, should be thought of as recording traces. By the defining the trace correctly, we have also avoided issues with torsion --- the ring elements never mix with $\mathbb{Z}$ at inopportune moments.

Before closing this section, we note that the universal Witt polynomials can be computed quite quickly using a combination of the presentation of $\overline{W}(A)$ and the procedure above.

\section{Higher Structures}\label{sec:conjectures}

In this section, we make some conjectures and observations that we will return to in future work. Since nothing is proved in this section, we will content ourselves with sketches. The reader who, for mysterious reasons, is not interested in higher algebraic $K$-theory or topological Hochschild homology may comfortably skip this section. 

The group $\widetilde{K}_0 (\End(A))$ may be realized as the group of components of an algebraic $K$-theory spectrum. To see this, we note that the category of endomorphisms is appropriate for input into one of the algebraic $K$-theory machines: the category $\End(A)$ is an exact or Waldhausen category in the sense of \cite{quillen} or \cite{waldhausen}. What is more surprising is that this is true of the $V$-filtration as well.

\begin{prop}
Let $V^{\geq k} \End(A)$ be the thick subcategory generated by images of $V^i$ with $i \geq k$. Then $V^{\geq k} \End(A)$ is a Waldhausen category. 
\end{prop}

Thus, $V^{\geq k} \End(A)$ may be fed into an algebraic $K$-theory machine. This entitles us to the following surprising definition. 

\begin{defn}
  The \textbf{$V$-filtration on the algebraic $K$-theory of endomorphisms} is given by 
  \[
  V^{\geq k} K(\End(A))  := K(V^{\geq k} \End(A)) := \Omega |S_\bullet V^{\geq k} \End(A)|
  \]
\end{defn}

Taking algebraic $K$-theory of the filtered category, we obtain maps of spectra
\[
\begin{tikzcd}
\cdots \ar{r} & V^{\geq 2} K(\End(A)) \ar{r} & V^{\geq 1} K(\End(A)) \ar{r}{\cong} & K(\End(A))
\end{tikzcd}
\]
which upon taking homotopy cofibers gives a pro-spectrum
\[
\begin{tikzcd}
  \ast & K(\End(A))/V^2 \ar{l} & K(\End(A))/V^3\ar{l} & K(\End(A))/V^4\ar{l} & \cdots \ar{l}.
\end{tikzcd}
\]

In analogy with the case of $K_0$, we make the following definition. 

\begin{defn}
  The \textbf{$V$-completion} of $\widetilde{K} (\End(A))$ is
  \[
  \holim \widetilde{K}(\End(A))/V^j 
  \]
\end{defn}

Based on the work above, it seems reasonable to make the following conjecture
\begin{conj}
  There is a weak equivalence of spectra
  \[
  \begin{tikzcd}
    \holim \widetilde{K}(\End(A))/V^j \ar{r} & \TR(A) 
  \end{tikzcd}
  \]
  where $\TR(A)$ is topological restriction homology. 
\end{conj}

That is, we conjecture that  $\TR(A)$ can be constructed directly from $\widetilde{K}(\End(A))$ without reference to topological Hochschild homology. One can also show that $\widetilde{K}(\End(A))$ is an $S^1$-ring spectrum, and so the homotopy groups acquire a differential. It should be that the homotopy groups of $V$-completion of $\widetilde{K}(\End(A))$ are very closely related to the big de Rham-Witt complex \cite{hesselholt_big_de_rham}.

A related, easier, conjecture is the following. 

\begin{conj}
There is a weak equivalence of spectra
\[
\begin{tikzcd}
\holim \widetilde{K} (\Set^{\text{fin}}_{\cong})/V^j \ar{r} & \TR(S)
\end{tikzcd}
\]
where $S$ denotes the sphere spectrum. This lifts the isomorphism $\varprojlim K_0 (\operatorname{Set}^{\text{fin}})/V^j \to W(\mathbb{Z})$. 
\end{conj}

More speculatively, one could posit that $K(\End(A))/V$ is a more canonical version of $\THH$, or perhaps equivalent to $\THH$. In fact, $K(\End(A))$ is a cyclotomic spectrum \footnote{We show this in forthcoming work with Lind, Malkiewich, Ponto and Zakharevich}, and we conjecture the following. 

\begin{conj}
There is a weak equivalence of $S^1$-spectra
\[
\begin{tikzcd}
\widetilde{K}(\End(A))/V^2 = V^{\geq 1} \widetilde{K}(\End(A)) / V^{\geq 2} \widetilde{K}(\End(A))  \ar{r} & \THH(A)
\end{tikzcd}
\]
and the topological Dennis trace is given by the truncation map
\[
\begin{tikzcd}
  \widetilde{K}(\End(A)) \ar{r} &  V^{\geq 1} \widetilde{K}(\End(A))/ V^{\geq 2} \widetilde{K}(\End(A))
\end{tikzcd}
\]
\end{conj}

\begin{rmk}
This is a ``derived'' version of the trace map defined for $K_0$ above. 
\end{rmk}

Exactly what $\THH$ ``is'' has always been a been of a mystery to the author. Among other reasons, it could be that $\THH$ is simply the first truncation of a much more natural object --- $\widetilde{K}(\End(A))$. We also point out that this is expressing something like the fact that $\THH(A)$ is the ``tangent space'' of $\widetilde{K}(\End(A))$. Such intuition has been in the air for a very long time, and it would be pleasant to make it precise via this perspective.

Finally, for general Waldhausen categories $\mathcal{C}$ there is no good definition of $\thh$: the input for $\thh$ is either a spectrally enriched category, or a category with enough structure that it can be spectrally enriched (see, e.g. \cite{dundas_mccarthy, blumberg_mandell}). But there \textit{is} still a Verschiebung map on $\operatorname{End}(\mathcal{C})$ for any Waldhausen category. Thus, we propose the following.

\begin{conj}
  Let $\mathcal{C}$ be a Waldhausen category. The definition of $\thh(\mathcal{C})$ should be
  \[
  \THH(\mathcal{C}) := \widetilde{K}(\End(\mathcal{C}))/V^{\geq 2}
  \]
\end{conj}

We are addressing these conjectures in work in progress.

\bibliographystyle{plain}
\bibliography{witt_bib}

\end{document}